\documentclass{amsart}

\usepackage{tikz}
\usepackage{amsmath}
\usepackage{amsthm}
\usepackage{amssymb}
\usepackage{mathtools}

\newcommand{\Q}{\mathbb{Q}}

\newcommand{\R}{\mathbb{R}}
\newcommand{\N}{\mathbb{N}}

\newcommand{\dom}{\operatorname{dom}}

\renewcommand{\P}{\mathcal{P}}
\newcommand{\M}{\mathcal{M}}
\newcommand{\halts}{\hspace{-0.1cm}\downarrow}

\newcommand{\one}{\mathbf{1}}

\newcommand{\ol}{\overline}

\newtheorem{theorem}{Theorem}[section]
\newtheorem{lemma}[theorem]{Lemma}

\newtheorem{corollary}[theorem]{Corollary}

\newtheorem{proposition}[theorem]{Proposition}

\theoremstyle{definition}
\newtheorem{definition}[theorem]{Definition}
\newtheorem{question}[theorem]{Question}

\newtheorem{example}[theorem]{Example}

\numberwithin{equation}{section}

\begin{document}
\title{Effective Notions of Weak Convergence of Measures on the Real Line}
\author{Timothy H. McNicholl}
\address{Department of Mathematics\\
Iowa State University\\
Ames, Iowa 50011}
\email{mcnichol@iastate.edu}
\author{Diego A. Rojas}
\address{Department of Mathematics\\
Iowa State University\\
Ames, Iowa 50011}
\email{darojas@iastate.edu}

\begin{abstract}
We establish a framework for the study of the effective theory of weak convergence of measures.  
We define two effective notions of weak convergence of measures on $\R$: one uniform and one non-uniform. 
We show that these notions are equivalent.   By means of this equivalence, we prove an effective version of the Portmanteau Theorem, which consists of multiple equivalent definitions of weak convergence of measures.
\end{abstract}

\maketitle

\section{Introduction}\label{sec:intro}

Beginning with Alan Turing's pioneering work on the computability of real numbers in the mid-1930s \cite{T37}, mathematicians have understood that concepts central to the study of mathematical analysis have analogues in computability theory. 
Computable analysis examines these computable analogues and attempts to recover as much of the classical theory of analysis as possible while working in a framework for computation. 
Current research in the field has investigated the interactions between computability and measure theory, one of the main sub-fields of classical analysis, especially with its deep connections to algorithmic randomness \cite{R16, R18}. 
In particular, one aspect of measure theory in algorithmic randomness lies in the connections between random points and measures defined as limits of sequences of measures under different convergence notions \cite{CL18, H11}.

The most commonly studied notion of convergence for sequences of measures in analysis is \emph{weak convergence}. 
A sequence of finite Borel measures $\{\mu_n\}_{n\in\N}$ on a separable metric space $X$ \emph{weakly converges} to a measure $\mu$ if, for every bounded continuous real-valued function $f$ on $X$, the sequence of real numbers $\lim_n \int_X f\ d\mu_n = \int_X f\ d\mu$. 
The measure $\mu$ in this definition is unique and is called the \emph{weak limit measure}.  Weak convergence is a useful tool in probability.   In particular, weak convergence can be used to demonstrate convergence in distribution and to construct stable distributions \cite{Durrett.2010}.

Weak convergence also underlies current research in computable measure theory and algorithmic randomness.  
In particular, the space $\M(X)$ of finite Borel measures on a computable metric space $X$ forms a computable metric space when equipped with the Prokhorov metric \cite{HR09b}.  The Prokhorov metric was introduced by Prokhorov \cite{P56} in 1956 and metrizes the topology of weak convergence.
Galatolo, Hoyrup, and C. Rojas proved an effective version of the Borel-Cantelli  Lemma by using the computability framework obtained by weak convergence of probability measures \cite{GHR09}. 
A few years later, Binder, Braverman, C. Rojas, and Yampolsky used this framework to show that a Brolin-Lyubich probability measure is always computable \cite{BBRY11}. 
In algorithmic randomness, G\'{a}cs developed an approach to uniform randomness tests using the properties of $\P(X)$ (the space of Borel probability measures on $X$) equipped with the weak convergence topology \cite{G05}. 

As shown in Example \ref{ex:not.ewc} below, a weakly convergent sequence of measures, even if computable, need not have a computable limit.  
This possibility has been noted before.  
More specifically, the First Computability Obstruction (Proposition 1.1.3, \cite{H14}) states that any sequence of measures generated by iterating a computable measure through a cellular automaton converges weakly to a $\Delta^0_2$-computable measure.  In fact, it is a fairly easy consequence of our results that any sequence of measures that is computable and weakly convergent has a $\Delta^0_2$ limit.
The conditions under which the weak limit of a computable sequence of measures is computable appear not to have been investigated.  We are thus led to the following.
\begin{question}
What is a suitable definition of effective weak convergence?
\end{question}

We propose two answers to this question, one of which is uniform (Definition \ref{def:uew}), and one of which is not (Definition \ref{def:ew}).  At first glance, the uniform version is much stronger.  In particular, the uniform condition automatically implies weak convergence, whereas the non-uniform version does not.  However, our first main result (Theorem \ref{thm:ewc.equiv}) is that in fact the two definitions are equivalent.  

We then provide evidence that our definition of effective weak convergence in $\M(\R)$ is the appropriate computable analogue to classical weak convergence in $\M(\R)$. 
In particular, our other main result is an effective version of the Portmanteau Theorem, a 1941 theorem due to Alexandroff \cite{A41} that characterizes weak convergence of measures, for $\M(\R)$. 

This paper is divided as follows. 
Section 2 consists of the necessary background in both classical and computable analysis and measure theory. 
Section 3 covers some preliminary material used in latter sections to state and prove the main results of this paper.
In Section 4, we define effective notions of weak convergence in $\M(\R)$, and we show that they are equivalent for uniformly computable sequences in $\M(\R)$. 
In Section 5, we prove an effective version of the Portmanteau Theorem for $\M(\R)$. 
We conclude in Section 6 with a discussion on the results and implications for future research in this direction.

\section{Background}\label{sec:back}

\subsection{Background from classical analysis and measure theory}

We denote the set of all bounded continuous functions on $\R$ by $C_b(\R)$.  
We denote the indicator function of a set $E$ by $\one_E$.

Fix a measure $\mu\in\M(\R)$, and let $A \subseteq \R$.   $A$ is said to be a \emph{$\mu$-continuity set} if $\mu(\partial A)=0$. 

Below, we state a version of the classical Portmanteau Theorem. 
Although there are as many as ten equivalent definitions of weak convergence (see \cite{LKN16}), we will focus on five.

\begin{theorem}[Classical Portmanteau Theorem, \cite{A41}]\label{thm:cpt}
Let $\{\mu_n\}_{n\in\N}$ be a sequence in $\M(\R)$. The following are equivalent.
	\begin{enumerate}
    		\item $\{\mu_n\}_{n\in\N}$ weakly converges to $\mu$.
    		\item For every uniformly continuous $f\in C_b(\R)$,
    		\[\hspace{0.5in}
    		\lim_{n\rightarrow\infty}\int_{\R}fd\mu_n=\int_{\R}fd\mu.
    		\]
    		\item For every closed $C\subseteq\R$, 
			\[\hspace{0.5in}
			\limsup_{n\rightarrow\infty}\mu_n(C)\leq\mu(C).
			\]
    		\item For every open $U\subseteq\R$, 
			\[\hspace{0.5in}
			\liminf_{n\rightarrow\infty}\mu_n(U)\geq\mu(U).
			\]
    		\item For every $\mu$-continuity set $A\subseteq\R$, 
			\[\hspace{0.5in}
			\lim_{n\rightarrow\infty}\mu_n(A)=\mu(A).
			\]
	\end{enumerate}
\end{theorem}

\subsection{Background from computable analysis and computable measure theory}

We presume familiarity with the fundamentals of computability theory as expounded in \cite{Cooper.2004}.  
For a more expansive treatment of computable analysis, see \cite{Weihrauch.2000}.

We say that a bounded interval $I\subseteq\R$ is \emph{rational} if each of its endpoints is rational.  Fix an effective enumeration $\{I_i\}_{i\in\N}$ of the set of all rational open intervals. 

If $I$ is compact, let $P_{\Q}(I)$ denote the space of polygonal functions on $I$ with rational vertices; we will refer to these functions as \emph{rational polygonal functions} on $I$.
When $p\in P_{\Q}[a,b]$,  we extend $p$ to all of $\R$ by letting $p(x)=p(a)$ for $x<a$ and $p(x)=p(b)$ for $x>b$.

An open set $U\subseteq\R$ is $\Sigma^0_1$ if $\{i\in\N:I_i\subseteq U\}$ is c.e.. 
Similarly, a closed set $C\subseteq\R$ is $\Pi^0_1$ if $\{i\in\N:I_i\cap C=\emptyset\}$ is c.e..
We denote the set of $\Sigma^0_1$ subsets of $\R$ by $\Sigma^0_1(\R)$, and we denote the set of $\Pi^0_1$ subsets of $\R$ by $\Pi^0_1(\R)$.
We say that $e \in \N$ \emph{indexes} $U \in \Sigma_1^0(\R)$ if $e$ indexes $\{i \in \N\ :\ I_i \subseteq U\}$.  
Indices of sets in $\Pi^0_1(\R)$ are defined analogously.  A pair $(U,V)$ of $\Sigma_1^0$ sets is indexed by an $e \in \N$ if $U$ is indexed by $(e)_0$ and $V$ is indexed by $(e)_1$.

Fix a real number $x$.  A \emph{Cauchy name} of $x$ is a sequence $\{q_n\}_{n \in \N}$ of rational numbers so that 
$\lim_n q_n = x$ and so that $|q_n - q_{n+1}| < 2^{-n}$ for all $n \in \N$.  

When $I\subseteq\R$ is a compact rational interval and $J\subseteq\R$ is a rational open interval, we let $N_{I,J}=\{f\in C(\R):f[I]\subseteq J\}$. 
A \emph{compact-open (c.o.)} name of a function $f\in C(\R)$ is an enumeration of $\{N_{I,J}:f\in N_{I,J}\}$.  
If $f \in C(\R)$, then $f$ is computable if and only $f$ has a computable $c.o.$ name.

Each of the names we have just discussed can be represented as a point in $\Sigma^\omega$ for a sufficiently large alphabet $\Sigma$.

Fix $x \in \R$.  $x$ is \emph{computable} if it has a computable Cauchy name.  An index of such a name is also said to be an index of $f$.  $x$ is \emph{left-c.e.} (\emph{right-c.e.}) if its left (right) Dedekind cut is c.e..  It follows that $x$ is computable if and only if it is left-c.e. and right-c.e..  

A sequence $\{x_n\}_{n \in \N}$ is computable if $x_n$ is computable uniformly in $n$.

A function $f:\subseteq\R\rightarrow\R$ is \emph{computable} if there is a Turing functional $F$ so that
$F(\rho)$ is a Cauchy name of $f(x)$ whenever $\rho$ is a Cauchy name of $x$.  An index of such a functional $F$ is also said to be an index of $f$.  We denote the set of all bounded computable functions on $\R$ by $C_b^c(\R)$.

A function $f: \subseteq \R \rightarrow \R$ is \emph{lower semi-computable} if there is a Turing functional $F$ so that 
$F(\rho)$ enumerates the left Dedekind cut of $f(x)$ whenever $\rho$ is a Cauchy name of $x$.
A function $f: \subseteq \R \rightarrow \R$ is \emph{upper semi-computable} if $-f$ is lower semi-computable.

A function $F:\subseteq C(\R)\rightarrow\R$ is \emph{computable} if there is a Turing functional $\Phi$ 
so that $\Phi(\rho)$ is a Cauchy name of $F(f)$ whenever $\rho$ is a c.o.-name of $f$.  
An index of such a functional $\Phi$ is also said to be an index of $F$.

Suppose $\{a_n\}_{n \in \N}$ is a convergent sequence of reals, and let $a = \lim_n a_n$.  A \emph{modulus of convergence} of $\{a_n\}_{n \in \N}$ is a function $g : \N \rightarrow \N$ so that $|a_n - a| < 2^{-k}$ whenever $k \geq g(n)$.  

A measure $\mu\in\M(\R)$ is \emph{computable} if $\mu(\R)$ is a computable real and $\mu(U)$ is left-c.e. uniformly in an index of $U\in\Sigma^0_1(\R)$;  i.e. it is possible to compute an index of the left Dedekind cut of $\mu(U)$ from an index of $U$.  A sequence of measures $\{\mu_n\}_{n\in\N}$ in $\M(\R)$ is \emph{uniformly computable} if $\mu_n$ is a computable measure uniformly in $n$.

Suppose $\mu \in \M(\R)$ is computable.  A pair $(U,V)$ of $\Sigma^0_1$ subsets of $\R$ is \emph{$\mu$-almost decidable} if $U\cap V=\emptyset$, $\mu(U\cup V)=\mu(\R)$, and $\ol{U\cup V}=\R$.  
If, in addition, $U \subseteq A \subseteq \R - V$, then we say that $(U,V)$ is a $\mu$-almost decidable pair of $A$.  
We then say $A$ is \emph{$\mu$-almost decidable} if it has a $\mu$-almost decidable pair. 
Suppose $(U,V)$ is a $\mu$-almost decidable pair of $A$. Then, $e$ indexes $A$ if $e=\left<i,j\right>$ for some index $i$ of $U$ and some index $j$ of $V$.  We note that $\mu$-almost decidable sets are effective analogues of $\mu$-continuity sets.  The definition of $\mu$-almost decidable set is from \cite{R18}.

\section{Preliminaries}\label{sec:prelim}

In order to formulate an effective Portmanteau Theorem, we require an effective definition of $\limsup$ and $\liminf$. 
The following definitions provide such an effectivization.

\begin{definition}\label{def:wit.lwr.upr}
Suppose $\{a_n\}_{n \in \N}$ is a sequence of reals, and let $g : \subseteq \Q \rightarrow \N$.
\begin{enumerate}
	\item We say $g$ \emph{witnesses that $\liminf_n a_n$ is not smaller than $a$} if $\dom(g)$ is the left Dedekind cut of $a$ and if $r < a_n$ whenever $r \in \dom(g)$ and $n \geq g(r)$.
	
	\item We say $g$ \emph{witnesses that $\limsup_n a_n$ is not larger than $a$} if $\dom(g)$ is the right Dedekind cut of $a$ and if $r > a_n$ whenever $r \in \dom(g)$ and $n \geq g(r)$.
\end{enumerate}
\end{definition}

We observe that $\liminf_n a_n \leq a$ if and only if there is a witness that $\liminf_n a_n$ is not smaller than $a$.   
Similarly, $\limsup_n a_n \leq a$ if and only if there is a witness that $\limsup_n a_n$ is not larger than $a$. 

The following proposition effectivizes the fact that for any sequence of reals $\{a_n\}_{n\in\N}$, $\limsup_n a_n\leq a \leq \liminf_n a_n$ implies $\lim_n a_n=a$. 

\begin{proposition}\label{prop:wit.2.mod}
Suppose there is a computable witness that $\liminf_n a_n$ is not smaller than $a$, and suppose there is a computable
witness that $\limsup_n a_n$ is not larger than $a$.  Then, $\lim_n a_n = a$, and $\{a_n\}_{n \in \N}$ has a computable modulus of convergence.
\end{proposition}

\begin{proof}
We first note that $\limsup_n a_n \leq a \leq \liminf_n a_n$ and so $\lim_n a_n= a$.  We also note that $a$ is computable since its left and right Dedekind cuts are c.e..

Let $g_1$ be a computable witness that $\liminf_n a_n$ is not smaller than $a$, and let 
$g_2$ be a computable witness that $\limsup_n a_n$ is not larger than $a$.  Define $g : \N \rightarrow \N$ as follows.
Let $k \in \N$.  Since $a$ is computable, we can compute $r_1, r_2 \in \Q$ so that $r_1 < a < r_2$ and $r_2 - r_1 < 2^{-k}$.  Let $g(k) = \max\{g_1(r_1), g_2(r_2)\}$.  It follows from Definition \ref{def:wit.lwr.upr} that $|a_n - a| < 2^{-k}$ when $n \geq g(k)$.
\end{proof}

We note that the proof of Proposition \ref{prop:wit.2.mod} is uniform.

Corollary 4.3.1. in \cite{HR09b} gives us a way to characterize computable measures in terms of their integrals. 
The following proposition is a useful extension of this characterization.

\begin{proposition}\label{prop:comp.equiv}
For $\mu\in\M(\R)$, the following are equivalent.
\begin{enumerate}
	\item $\mu$ is computable.\label{prop:comp.equiv::mc}
	
	\item $f\mapsto\int_{\R}f\ d\mu$ is computable on nonnegative $f\in C^c_b(\R)$.  That is, from an index of an $f \in C^c_b(\R)$ it is possible to compute an index of $\int_\R f\ d\mu$.  \label{prop:comp.equiv::ic}
	
	\item $f\mapsto\int_{\R}f\ d\mu$ is computable on uniformly continuous $f\in C^c_b(\R)$ such that $0\leq f\leq1$.
	\label{prop:comp.equiv::iucc}
\end{enumerate}
\end{proposition}

\begin{proof}
Clearly (\ref{prop:comp.equiv::ic}) implies (\ref{prop:comp.equiv::iucc}), and (\ref{prop:comp.equiv::mc}) implies (\ref{prop:comp.equiv::ic}) by Corollary 4.3.1 in \cite{HR09b}.

We have left to show that (\ref{prop:comp.equiv::iucc}) implies (\ref{prop:comp.equiv::mc}). 
To this end, suppose $f\mapsto\int_{\R}fd\mu$ is computable on uniformly continuous $f\in C^c_b(\R)$ such that $0\leq f\leq1$. 
Note that $\mu(\R)$ is a computable real since the constant function $x\mapsto1$ is computable and uniformly continuous. 

Now, fix $U\in\Sigma^0_1(\R)$.
It suffices to show that $\mu(U)$ is the limit of a non-decreasing sequence of left-c.e. reals uniformly in an index of $U$. 
Let $U_n=\bigcup\{I_i: 0\leq i\leq n\ \wedge\ I_i \subseteq U\}$. 
We note that, since $U$ is open, $U= \bigcup_n U_n$.  The sequence $\{\mu(U_n)\}_{n \in \N}$ is non-decreasing, and by continuity of measure, $\lim_n \mu(U_n) = \mu(U)$.  Thus, it suffices to show that $\mu(U_n)$ is left-c.e. uniformly in $n$ and an index of $U$.
To this end, we observe that for each $n\in\N$, 
\[
\one_{U_n}=\max\{\one_{I_i}:0\leq i\leq n, I_i\subseteq U\}
\] 
is nonnegative and lower-semicomputable uniformly in $n$. 
Thus, we can compute for each $n\in\N$ a sequence of computable Lipschitz functions $0\leq t_{n,k}\leq1$ such that $t_{n,k}$ increases to $\one_{U_n}$ pointwise as $k\rightarrow\infty$ (see Proposition C.7, \cite{G05}).
By the Monotone Convergence Theorem, $\mu(U_n)=\lim_{k\rightarrow\infty}\int_{\R}t_{n,k}\mbox{ }d\mu$.  
Since $\int_\R t_{n,k} d\mu$ is computable uniformly in $n,k$, $\mu(U_n)$ is left-c.e. uniformly in $n$ and an index of $U$.  
\end{proof}

\section{Effectivizing Weak Convergence of Measures on $\R$}

Throughout this section, let $\{\mu_n\}_{n\in\N}$ be a sequence in $\M(\R)$, and fix $\mu\in\M(\R)$.

To devise an effective notion of weak convergence for $\{\mu_n\}_{n\in\N}$, we need to impose a computability condition on the convergence of the sequence $\{\int_{\R}fd\mu_n\}_{n\in\N}$ with $f\in C_b(\R)$.  Our first attempt is the following.

\begin{definition}\label{def:ew}
We say $\{\mu_n\}_{n\in\N}$ \emph{effectively weakly converges} to $\mu$ if for every $f \in C_b^c(\R)$, 
$\lim_n \int_\R f\ d\mu_n = \int_\R f\ d\mu$ and it is possible to compute an index of a modulus of convergence for 
$\{\int_\R f\ d\mu_n\}_{n \in \N}$ from an index of $f$ and a bound $B \in \N$ on $|f|$. 
\end{definition}

Definition \ref{def:ew} at first glance seems reasonable.  However, as it only produces moduli of convergence for computable functions, it does not automatically imply weak convergence.  One generally expects that the 
effective version of a classical notion to imply the classical notion.  Thus, we are led to also consider the following.

\begin{definition}\label{def:uew}
We say $\{\mu_n\}_{n\in\N}$ \emph{uniformly effectively weakly converges} to $\mu$ if it weakly converges to $\mu$ 
and there is a uniform procedure that for any $f \in C_b(\R)$ computes a modulus of convergence for 
$\{\int_\R f\ d\mu_n\}_{n \in \N}$ from a c.o.-name of $f$ and a bound $B \in \N$ on $|f|$.
\end{definition}

While effective weak convergence requires computable real-valued functions, uniform effective weak convergence may use \emph{any} real-valued function. As such, a modulus of uniform effective weak convergence is able to compute with incomputable information. Hence, uniformly effective weak convergence is a stronger convergence notion than effective weak convergence.  In particular, since every function in $C_b(\R)$ has a c.o. name, it is automatic that every uniformly effectively weakly convergent sequence weakly converges.  
On the other hand, since it avoids the use of names, Definition \ref{def:ew} seems easier to work with in practice.  
Fortunately, we have the following.

\begin{theorem}\label{thm:ewc.equiv}
Suppose $\{\mu_n\}_{n\in\N}$ is uniformly computable. The following are equivalent.
	\begin{enumerate}
		\item $\{\mu_n\}_{n\in\N}$ is effectively weakly convergent.
		\item $\{\mu_n\}_{n\in\N}$ is uniformly effectively weakly convergent.
	\end{enumerate}
\end{theorem}

Before proving Theorem \ref{thm:ewc.equiv}, we prove a few preliminary results.  The first demonstrates one of the desired properties of effective weak convergence.  

\begin{proposition}\label{prop:uew}
If $\{\mu_n\}_{n\in\N}$ is uniformly computable and effectively weakly converges to $\mu$, then $\mu$ is a computable measure.
\end{proposition}

\begin{proof}
By Proposition \ref{prop:comp.equiv}, it suffices to show that $f\mapsto\int_{\R}fd\mu$ is computable on the set of nonnegative $f\in C_b^c(\R)$.  Let $f \in C^c_b(\R)$.  Then, $\{\int_{\R}f\ d\mu_n\}_{n\in\N}$ is a computable sequence of reals since $\{\mu_n\}_{n\in\N}$ is uniformly computable.  
Since $\{\mu_n\}_{n \in \N}$ effectively weakly converges to $\mu$, $\lim_n \int_\R f\ d\mu_n = \int_\R f\ d\mu$ and it is possible to compute a modulus of convergence for 
$\{\int_{\R}f\ d\mu_n\}_{n\in\N}$ from an index of $f$.  By Theorem 4.2.3 in \cite{Weihrauch.2000}, $\int_\R f\ d\mu$ is computable, and it is possible to compute an index of $\int_\R f\ d\mu$ from an index of $f$.
\end{proof}

Next, we need the following lemmas.

\begin{lemma}\label{lm:big.a}
Suppose $\{\mu_n\}_{n \in \N}$ is uniformly computable and effectively weakly converges to $\mu$.  From $N \in \N$, it is possible to compute $a,n_0 \in \N$ so that $\mu_n(\R \setminus [-a,a]) < 2^{-N}$ for all $n \geq n_0$.
\end{lemma}

\begin{proof}
First, define the sequence $\{w_{a,k}\}_{k\in\N}$ of functions on $\R$ by
\[
w_{a,k}(x)=\one_{\R\setminus[-a,a]}(x)+(2^k(x-a)+1)\one_{[a-2^{-k},a]}(x)+(-2^k(x+a)+1)\one_{[-a,-a+2^{-k}]}(x).
\]
The graph of $w_{a,k}$ is shown in Figure \ref{fig:wak}.
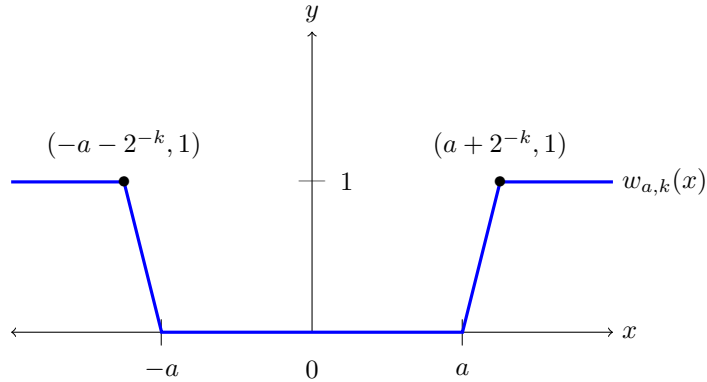
\begin{figure}[h!]
    \centering
    \begin{tikzpicture}
    \draw[<->] (-4,0) -- (4,0) node[right] {$x$};
    \draw[->] (0,0) -- (0,4) node[above] {$y$};
    \node at (0,-0.5) {$0$};
    \node at (-2,-0.5) {$-a$};
    \node at (2,-0.5) {$a$};
    \node at (-2,0) {$|$};
    \node at (2,0) {$|$};
    \node at (0,2) {---};
    \node[below,right] at (0.25,2) {$1$};
    \node at (-2.5,2.5) {$(-a-2^{-k},1)$};
    \node at (2.5,2.5) {$(a+2^{-k},1)$};
    \draw[-, very thick, blue] (-4,2) -- (-2.5,2) -- (-2,0) -- (2,0) -- (2.5,2) -- (4,2);
    \node at (-2.5,2) {\textbullet};
    \node at (2.5,2) {\textbullet};
    \node[right] at (4,2) {$w_{a,k}(x)$};
    \end{tikzpicture}
    \caption{The graph of $y=w_{a,k}(x)$ for fixed $a,k\in\N$. }
    \label{fig:wak}
\end{figure}

We note that for all $a \in \N$, 
\[
\one_{\R \setminus [-(a + 1), a + 1]} \leq w_{a,0} \leq \one_{\R \setminus [-a,a]}.
\]
Thus, for all $\nu \in \M(\R)$ and $a \in \N$, 
\[
\nu(\R \setminus [-(a + 1), a + 1])  \leq \int_\R w_{a,0}\ d\nu \leq \nu(\R \setminus [-a,a]).
\]
We first search for $a'$ so that $\int_\R w_{a',0}\ d\mu < 2^{-N}$.  Since $\mu$ is finite, it follows that 
this search must terminate.  Since $\mu$ is computable, this search is effective.  Set $a = a' + 1$.  Since 
$\{\mu_n\}_{n \in \N}$ effectively weakly converges to $\mu$, we can now compute an $n_0 \in \N$ so that 
$\int_\R w_{a',0} d\mu_n < 2^{-N}$ for all $n \geq n_0$.  Thus, 
$\nu_n(\R \setminus [-a,a]) < 2^{-N}$ for all $n \geq n_0$.
\end{proof}

\begin{lemma}\label{lm:comp.psi}
Suppose $\{\mu_n\}_{n \in \N}$ is uniformly computable and effectively weakly converges to $\mu$.  From 
a c.o.-name of an $f \in C_b(\R)$ and $N,B \in \N$ so that $|f| \leq B$, it is possible to compute $a,n_1 \in \N$ and $\psi \in P_\Q[-a,a]$ so that $|\int_\R (f - \psi) d\mu| < 2^{-N}$ and so that $|\int_\R (f - \psi) d\mu_n| < 2^{-N}$ whenever $n \geq n_1$.
\end{lemma}

\begin{proof}
By Lemma \ref{lm:big.a}, we can first compute $a, n_0 \in \N$ so that $\mu(\R \setminus [-a,a]) < B^{-1}2^{-(N+2)}$ and 
$\mu_n(\R \setminus [-a,a]) < B^{-1}2^{-(N+2)}$ for all $n \geq n_0$.

Define the function $T$ on $\R$ by
\[
T(x)=\one_{[-a,a]}(x)+(x+a+1)\one_{(-a-1,-a)}(x)+(-x+a+1)\one_{(a,a+1)}(x).
\]
The graph of $T$ is shown in Figure \ref{fig:tent}.
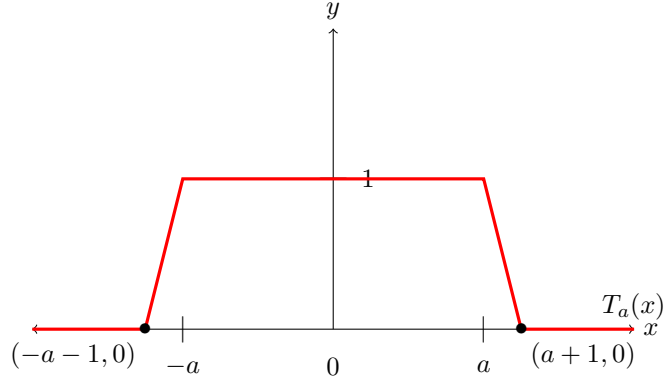
\begin{figure}[h!]
    \centering
    \begin{tikzpicture}
    \draw[<->] (-4,0) -- (4,0) node[right] {$x$};
    \draw[->] (0,0) -- (0,4) node[above] {$y$};
    \node at (0,-0.5) {$0$};
    \node at (-2,-0.5) {$-a$};
    \node at (2,-0.5) {$a$};
    \node at (-2,0) {$|$};
    \node at (2,0) {$|$};
    \node at (0,2) {---};
    \node[below,right] at (0.25,2) {$1$};
    \node[left] at (-2.5,-0.35) {$(-a-1,0)$};
    \node[right] at (2.5,-0.35) {$(a+1,0)$};
    \draw[-, very thick, red] (-4,0) -- (-2.5,0) -- (-2,2) -- (2,2) -- (2.5,0) -- (4,0);
    \node at (-2.5,0) {\textbullet};
    \node at (2.5,0) {\textbullet};
    \node[above] at (4,0) {$T_a(x)$};
    \end{tikzpicture}
    \caption{The graph of $y=T_a(x)$ for fixed $a\in\N$. }
    \label{fig:tent}
\end{figure}
Thus, $\nu([-a,a]) \leq \int_\R T d\nu$ for all $\nu \in \M(\R)$.

Since $\{\mu_n\}_{n \in \N}$ effectively weakly converges to $\mu$, we can now compute an $n_1 \geq n_0$ so that 
$|\int_\R T\ d\mu_n - \int_R T\ d\mu| < 1$ whenever $n \geq n_1$.

Fix $f \in C_b(\R)$.  From a c.o.-name of $f$, we can compute a c.o.-name of $f|_{[-a,a]}$.  Then, by means of Theorem 6.2.1 of \cite{Weihrauch.2000}, we can compute a $\psi \in P_\Q[-a,a]$ so that 
\[
\max\{|f(x) - \psi(x)|\ :\ x \in [-a,a]\} < \frac{2^{-(N+1)}}{1 + \int_\R T\ d\mu}.
\]
Let $n \in \N$, and suppose $n \geq n_1$.  By the construction of $T$, $\mu([-a,a]) < 1 + \int_R T\ d\mu$. However, 
since $n \geq n_1$, we also have
\begin{eqnarray*}
\mu_n([-a,a]) & \leq & \int_\R T\ d\mu_n\\
& \leq & \left|\int_\R T\ d\mu_n - \int_R T\ d\mu\right| + \int_\R T\ d\mu\\
& < & 1 + \int_\R T\ d\mu.
\end{eqnarray*}
If $\nu \in \M(\R)$, then 
\begin{eqnarray*}
\left| \int_\R (f - \psi)\ d\nu\right| & \leq & \int_{[-a,a]} |f - \psi|\ d\nu + \int_{\R\setminus [-a,a]} |f - \psi|\ d\nu\\
& \leq & \frac{2^{-(N+1)}}{1 + \int_\R T\ d\mu} \nu([-a,a]) + 2B \nu(\R\setminus [-a,a]).
\end{eqnarray*}
It follows that $|\int_R(f -\psi)\ d\mu| < 2^{-N}$ and that 
$|\int_\R (f - \psi)\ d\mu_n| < 2^{-N}$.
\end{proof}

\begin{proof}[Proof of Theorem \ref{thm:ewc.equiv}]
It is possible to compute a c.o.-name of $f \in C^c_b(\R)$ from an index of $f$.  It thus follows that 
every uniformly effectively weakly convergent sequence is effectively weakly convergent.

Now, suppose $\{\mu_n\}_{n\in\N}$ effectively weakly converges to $\mu$.  Let $B \in\N$, and suppose $\rho$ is a c.o. name of $f\in C_b(\R)$ with $|f|\leq B$.  We construct a function $G: \N \rightarrow \N$ as follows.
By means of Lemma \ref{lm:comp.psi}, we can compute $a, n_1 \in \N$ and $\psi \in P_\Q[-a,a]$ so that 
$|\int_\R (f - \psi)\ d\mu| < 2^{-(N+2)}$ and so that $|\int_\R (f - \psi)\ d\mu_n| < 2^{-(N+2)}$ when $n \geq n_1$.
Since $\{\mu_n\}_{n\in\N}$ effectively weakly converges to $\mu$, we can now compute an $n_2 \in \N$ so that 
$|\int_\R \psi\ d\mu_n - \int_R \psi\ d\mu| < 2^{-(N+1)}$ whenever $n \geq n_2$.  Set $G(N) = n_2$.

Suppose $n \geq G( N)$.  Then, 
\begin{align*}
\left|\int_{\R}fd\mu_n-\int_{\R}fd\mu\right|&\leq\left|\int_{\R}(f-\psi)d\mu_n\right|+\left|\int_{\R}\psi d\mu_n-\int_{\R}\psi d\mu\right|+\left|\int_{\R}(\psi-f)d\mu\right|\\
&<2^{-(N+2)}+2^{-(N+1)}+2^{-(N+2)}\\
&=2^{-N}.
\end{align*}
Thus, $G$ is a modulus of convergence for $\{\int_\R f\ d\mu_n\}_{n \in \N}$.   Since the construction of $G$ from $\rho$ and $B$ is uniform, $\{\mu_n\}_{n \in \N}$ uniformly effectively weakly converges to $\mu$.
\end{proof}

The corollary to Theorem \ref{thm:ewc.equiv} below follows from the observation that every uniformly effectively weakly convergent sequence in $\M(\R)$ weakly converges.

\begin{corollary}\label{cor:ewc.2.wc}
If a uniformly computable sequence of measures effectively weakly converges, then it weakly converges.
\end{corollary}

Finally, we note that the above results above allow us to distinguish the classical notion of weak convergence of measures on $\R$ from its effective versions.

\begin{example}\label{ex:not.ewc}
For $E\subseteq\R$, let $\mu_n(E)=\lambda(E\cap[0,q_n])$ for each $n$, where $\{q_n\}_{n\in\N}$ is a computable increasing sequence of rationals that converges to an incomputable left-c.e. $\alpha\in(0,1)$. Then it is easy to see that $\{\mu_n\}_{n\in\N}$ weakly converges to the measure $\mu(E)=\lambda(E\cap[0,\alpha])$ for $E\subseteq\R$. However, by Proposition \ref{prop:uew}, $\{\mu_n\}_{n\in\N}$ cannot effectively weakly converge since $\mu(\R)=\lambda(\R\cap[0,\alpha])=\lambda([0,\alpha])=\alpha$ is not computable.
\end{example}

%
%

\section{An Effective Portmanteau Theorem}

Below, we state an effective version of the classical Portmanteau Theorem (Theorem \ref{thm:cpt}).

\begin{theorem}[Effective Portmanteau Theorem]\label{thm:ept}
Let $\{\mu_n\}_{n\in\N}$ be a uniformly computable sequence in $\M(\R)$. The following are equivalent.
	\begin{enumerate}
   		\item $\{\mu_n\}_{n\in\N}$ effectively weakly converges to $\mu$.\label{thm:ept::ewc}
		
    		\item From $e,B \in \N$ so that $e$ indexes a uniformly continuous $f \in C_b(\R)$ with $|f| \leq B$, it is possible to compute a modulus of convergence of $\{\int_\R f\ d\mu_n\}_{n \in \N}$.\label{thm:ept::uc}
	
	    	\item $\mu$ is computable, and from an index of $C \in \Pi^0_1(\R)$ it is possible to compute an index of a 
		witness that $\limsup_n \mu_n(C)$ is not larger than $\mu(C)$.\label{thm:ept::clsd}
		
    		\item	$\mu$ is computable, and from an index of $U \in \Sigma^0_1(\R)$ it is possible to compute an index of 
		a witness that $\liminf_n \mu_n(U)$ is not smaller than $\mu(U)$.\label{thm:ept::opn}
		
    		\item $\mu$ is computable, and for every $\mu$-almost decidable $A$, $\lim_n \mu_n(A) = \mu(A)$ and an index of a modulus of convergence of $\{\mu_n(A)\}_{n \in \N}$ can be computed from a $\mu$-almost decidable index of $A$.\label{thm:ept::a.d.}
	\end{enumerate}
\end{theorem}

When $f \in C_b(\R)$ and $t \in \R$, let $U^f_t = \{f > t\}$, and let $\ol{U}^f_t = \{f \geq t\}$.  By a standard argument with Tonelli's Theorem, 
\[
\int_\R f\ d\nu = \int_0^1 \nu(U^f_t)\ dt = \int_0^1 \nu(\ol{U}^f_t)\ dt
\]
whenever $\nu \in \M(\R)$.

To prove Theorem \ref{thm:ept}, we will need the following lemmas.

\begin{lemma}\label{lm:sc.U}
Let $f\in C^c_b(\R)$ be so that $0<f<1$.  Fix a computable $\nu \in \M(\R)$.  
\begin{enumerate}
	\item The function $t \mapsto \nu(U^f_t)$ is lower-semicomputable uniformly in indices of $f$ and $\nu$. \label{lm:sc.U::lsc}
	\item The function $t \mapsto \nu(\ol{U}^f_t)$ is upper-semicomputable uniformly in indices of $f$ and $\nu$.
	\label{lm:sc.U::usc}
\end{enumerate}
\end{lemma}

\begin{proof}
Since $U^f_t = f^{-1}[(t,\infty)]$, $t \mapsto f^{-1}[(t,\infty)]$ is computable by Theorem 6.2.4 in \cite{Weihrauch.2000}.  That is, there is a Turing functional that computes an enumeration of $\{i\ :\ I_i \subseteq U^f_t\}$ from 
a Cauchy name of $t$.  It follows from Proposition 4.2.1 of \cite{HR09b} that $U \mapsto \nu(U)$ is lower semi-computable.  
That is, there is a Turing functional that for each open $U \subseteq \R$ computes an enumeration of the left Dedekind cut of $\nu(U)$ from an enumeration of $\{i\ :\ I_i \subseteq U\}$.  Thus, (\ref{lm:sc.U::lsc}).\\

By similar reasoning, $t \mapsto \nu (f^{-1}[(-\infty, t)])$ is lower semi-computable.  Since 
$\nu(\ol{U}_t^f) = \nu(\R) - \nu(f^{-1}[(-\infty,t)])$, it follows that $t\mapsto \nu(\ol{U}^f_t)$ is upper semi-computable.  Thus, (\ref{lm:sc.U::usc}).\\

By inspection, these arguments are uniform in indices of $\nu$ and $f$.
\end{proof}

\begin{lemma}\label{lm:mod.int}
Let $\{\mu_n\}_{n\in\N}$ be a uniformly computable sequence in $\M(\R)$ that weakly converges to a computable measure $\mu$.  Furthermore, let $f\in C_b^c(\R)$ be so that $0<f<1$.
\begin{enumerate}
	\item Suppose that from an index of $U \in \Sigma^0_1$ it is possible to compute an index of a witness that 
	$\liminf_n \mu_n(U)$ is not smaller than $\mu(U)$.  Then, there is a computable witness that 
	$\liminf_n \int_\R f\ d\mu_n$ is not smaller than $\int_\R f\ d\mu$.
	\label{lm:mod.int::lower}
	
	\item Suppose that from an index of $C \in \Pi^0_1$ it is possible to compute an index of a witness that 
	$\limsup_n \mu_n(C)$ is not larger than $\mu(C)$.  Then, there is a computable witness that 
	$\limsup_n \int_\R f\ d\mu_n$ is not larger than $\int_\R f\ d\mu$.
	\label{lm:mod.int::upper}
\end{enumerate}
\end{lemma}

\begin{proof}
(\ref{lm:mod.int::lower}): Let $\phi(t)=\mu(U^f_t)$ for all $t\in[0,1]$.  Thus, $\phi$ is non-increasing.  
Let $L$ denote the set of all tuples $(k,v_0,...,v_{2^k-1})$ so that $k\in\N$, $v_j\in\Q$, and $v_j<\phi((j+1)2^{-k})$ whenever $j < 2^k$.  Since $\phi$ is lower-semicomputable, $L$ is $L$ is c.e..  We note that if $(k, v_0, \ldots, v_{2^k - 1}) \in L$, then 
$v_j < \phi(t)$ whenever $j\cdot2^{-k} \leq t \leq (j+1)\cdot2^{-k}$.  We also note that if $(k, v_0, \ldots, v^{2^k - 1}) \in L$, then $\int_0^1 \phi(t)\ dt > 2^{-k} \sum_j v_j$.

We claim $\int_0^1\phi(t)dt>r$ if and only if there exists $(k,v_0,...,v_{2^k-1})\in L$ so that $\sum_{j<2^k}v_j>r2^k$. 
It suffices to prove the forward direction. 
Suppose $\int_0^1\phi(t)dt>r$. 
Since $\phi$ is non-increasing, $\phi$ is Riemann integrable. 
Let $\epsilon=2^{-1}\left(\int_0^1\phi(t)dt -r\right)$. 
Then, there exists $\delta>0$ so that $\left|S-\int_0^1\phi(t)dt\right|<\epsilon$ whenever $S$ is a Riemann sum for $\phi$ whose partition has width smaller than $\delta$. 
Let $k \in \N$ so that $2^{-k}<\delta$. 
Thus, 
\[
\left|\sum_{j<2^k} \phi((j+1)2^{-k}) 2^{-k} - \int_0^1\phi(t)dt\right|<\epsilon.
\]
Therefore, $\sum_{j<2^k} \phi((j+1)2^{-k}) 2^{-k}>r$. 
There exist rational numbers $v_0,...,v_{2^k-1}$ so that $v_j\leq \phi((j+1)2^{-k})$ and $\sum_{j<2^k} v_j 2^{-k}>r$. 
Thus, $(k,v_0,...,v_{2^k-1})\in L$. 

We now define a computable function $g : \subseteq \Q \rightarrow \N$.   
Given $r\in\Q$, search for $(k,v_0,...,v_{2^k-1}) \in L$ so that $\sum_{j<2^k} v_j 2^{-k}>r$. 
Then, for each $j < 2^k$, compute a witness $h_j$ that $\liminf_n \mu_n(U^f_{(j+1)2^{-k}})$ is not smaller than $\mu(U^f_{(j+1)2^{-k}})$ for $\{\mu_n(U^f_{(j+1)2^{-k}})\}_{n \in \N}$.
Set $g(r) = \max\{h_j(v_j)\ :\ j < 2^k\}$.  

Since $L$ is c.e., $g$ is a computable partial function. 
By what has just been shown, $g(r)\halts$ if and only if $\int_0^1\phi(t)dt>r$. 
Suppose $\int_0^1\phi(t)dt > r$ and $n\geq g(r)$. 
Then, $\mu_n(U^f_{j\cdot2^{-k}})>v_j$. 
Thus, $\mu_n(U^f_t)>v_j$ whenever $j \cdot 2^{-k} \leq t \leq (j+1)2^{-k}$. 
Hence, $\int_0^1\mu_n(U^f_t)dt > \sum_{j<2^k} v_j2^{-k} > r$.  
Therefore, $g$ witnesses that $\liminf_n \int_0^1 \mu_n(U^f_t)\ dt$ is not smaller than $\int_0^1 \mu(U^f_t)\ dt$.

(\ref{lm:mod.int::upper}):
 Let $\psi(t)=\mu(\ol{U}^f_t)$ for all $t\in[0,1]$.  Thus, $\psi$ is non-decreasing.
Let $H$ denote the set of all tuples $(k,u_0,...,u_{2^k-1})$ so that $k\in\N$, $u_j\in\Q$, and $u_j>\psi(j\cdot2^{-k})$ whenever $j\cdot2^{-k}\leq t\leq(j+1)2^{-k}$.  
Since $\psi$ is upper semi-computable, $H$ is c.e..  We note that if $(k,u_0,...,u_{2^k-1})\in H$, then 
$u_j > \psi(t)$ whenever $j\cdot 2^{-k} \leq t \leq (j+1) \cdot 2^{-k}$.  

We now claim $\int_0^1\psi(t)dt<r$ if and only if there exists $(k,u_0,...,u_{2^k-1})\in H$ so that $\sum_{j<2^k}u_j<r2^k$. 
The proof is a minor modification of the proof of the analogous claim in the proof of (\ref{lm:mod.int::lower}).  
The construction of the required witness now proceeds along the lines of the definition of $g$ in the proof of 
(\ref{lm:mod.int::lower}).
\end{proof}

We note that the proof of Lemma \ref{lm:mod.int} is uniform.

\begin{proof}[Proof of Theorem \ref{thm:ept}]
By Theorem \ref{thm:ewc.equiv}, (\ref{thm:ept::ewc}) implies (\ref{thm:ept::uc}).  The equivalence of (\ref{thm:ept::clsd}) and 
(\ref{thm:ept::opn}) follows by considering complements.\\

(\ref{thm:ept::uc}) $\Rightarrow$ (\ref{thm:ept::ewc}): Assume (\ref{thm:ept::uc}) holds. 
By Proposition \ref{prop:comp.equiv}, $\mu$ is a computable measure. 
Since rational polygonal functions are uniformly continuous, 
it follows from the proof of Theorem \ref{thm:ewc.equiv} that $\{\mu_n\}_{n \in \N}$ is uniformly effectively weakly convergent.  
Thus, by Theorem \ref{thm:ewc.equiv}, it follows that $\{\mu_n\}_{n\in\N}$ is effectively weakly convergent.\\

(\ref{thm:ept::ewc})$\Rightarrow$ (\ref{thm:ept::opn}): Assume (\ref{thm:ept::ewc}) holds. 
By Proposition \ref{prop:uew}, $\mu$ is computable.   

Let $U \in \Sigma^0_1(\R)$.  
We construct a function $g$ as follows.  
Given $r \in \Q$, first wait until $r$ is enumerated into the left Dedekind cut of $\mu(U)$.  
By means of Proposition C.7 of \cite{G05}, we can now compute a non-decreasing sequence $\{t_k\}_{k \in \N}$ of Lipschitz functions so that $0 \leq t_k \leq 1$ and so that $\lim_k t_k = \one_U$.  
 By the Monotone Convergence Theorem, $\lim_k \int_\R t_k\ d\mu = \mu(U)$.  Search for $k_0$ so that $\int_\R t_{k_0}\ d\mu > r$.  We then compute $N_0, n_0 \in \N$ so that 
$r + 2^{-N_0} < \int_\R t_{k_0}\ d\mu$ and 
$|\int_\R t_{k_0}\ d\mu - \int_\R t_{k_0}\ d\mu_n| < 2^{-N_0}$ when $n \geq n_0$.  Set $g(r) = n_0$.  Thus, when 
$n \geq g(r)$, $\mu_n(U) \geq \int_\R t_{k_0}\ d\mu_n > r$.  Therefore, $g$ witnesses that $\liminf_n \mu_n(U)$ is not smaller than $\mu(U)$.

Finally, we note that the construction of $g$ is uniform in that an index of $g$ can be computed from an index of $U$.  Thus, (\ref{thm:ept::opn}).\\

(\ref{thm:ept::opn}) $\Rightarrow$ (\ref{thm:ept::ewc}): Suppose (\ref{thm:ept::opn}).  
Thus, (\ref{thm:ept::clsd}).   

Fix $f \in C^c_b(\R)$, and suppose $B \in \N$ is an upper bound on $|f|$.  Set 
\[
h = \frac{f + B + 1}{2(B+1)}.
\]
Thus, $0 < h < 1$.  

Let $a_n = \int_\R h\ d\mu_n$.  
Let $a = \int_\R h\ d\mu$.  
By Lemma \ref{lm:mod.int}, there is a computable witness that 
$\liminf_n a_n$ is not smaller than $a$, and there is a computable witness that 
$\limsup_n a_n$ is not larger than $a$.  Thus, by Proposition \ref{prop:wit.2.mod}, $\lim_n a_n = a$ and 
$\{a_n\}_{n \in \N}$ has a computable modulus of convergence.  It follows that 
$\{\int_R f\ d\mu_n\}_{n \in \N}$ has a computable modulus of convergence.

The argument just given is uniform, and so we conclude $\{\mu_n\}_{n \in \N}$ effectively weakly converges to $\mu$.\\ 

(\ref{thm:ept::opn}) $\Rightarrow$ (\ref{thm:ept::a.d.}): Suppose (\ref{thm:ept::opn}).   Thus, (\ref{thm:ept::clsd}). 

Suppose $A$ is $\mu$-almost decidable.  Let $(U,V)$ be a $\mu$-almost decidable pair for $A$.  
Thus, $\mu(U) = \mu(A) = \mu(\R \setminus V)$.  

By (\ref{thm:ept::opn}), there is a computable witness that $\liminf_n \mu_n(U)$ is not smaller than $\mu(A)$; let $g_1$ be such a witness.  By (\ref{thm:ept::clsd}), there is a computable witness that $\limsup_n \mu_n(\R \setminus V)$ is not larger than $\mu(A)$; let $g_2$ be such a witness.

Since $\mu_n(U) \leq \mu_n(A) \leq \mu_n(\R \setminus V)$, $g_1$ is also a witness that $\liminf_n \mu_n(A)$ is not smaller than $\mu(A)$, and $g_2$ is also a witness that $\limsup_n \mu_n(A)$ is not larger than $\mu(A)$.  
So, by Proposition \ref{prop:wit.2.mod}, $\{\mu_n(A)\}_{n \in A}$ has a computable modulus of convergence $g$.   

The argument just given is uniform in that an index of $g$ can be computed from an index of $A$.  Hence, (\ref{thm:ept::a.d.}). \\

(\ref{thm:ept::a.d.}) $\Rightarrow$ (\ref{thm:ept::clsd}): Suppose (\ref{thm:ept::a.d.}).  Thus, $\mu$ is computable.  
Let $C \in \Pi^0_1(\R)$.  Set
\[
C_k = \bigcap \{\R \setminus I_i\ :\ i \leq k\ \wedge\ I_i \cap C = \emptyset\}.
\]
Thus, $C_k \supseteq C_{k+1}$ and $\bigcap_k C_k = C$.  Therefore, by continuity of measure, 
$\nu(C) = \lim_k \nu(C_k)$ for all $\nu \in \M(\R)$. 
Since $\mu$ is computable, $k \mapsto \mu(C_k)$ is upper semi-computable.  

We construct a function $g$ as follows.  Let $r \in \Q$.  Wait until $r$ is enumerated into the right Dedekind cut of $\mu(C)$.  Then, compute $k_0$ so that $\mu(C_{k_0}) < r$.  

For every $R > 0$, let:
\begin{eqnarray*}
\ol{B}_R & = & \{x \in \R\ :\ d(x, C_{k_0}) \leq R\}\\
C_R & =& \{x \in \R\ :\ d(x, C_R) = R\}
\end{eqnarray*}
By a standard counting argument, for every open interval $I \subseteq (0, \infty)$, there exists $R \in I$ so that 
$\mu(C_R) = 0$.  Also, $C_R \supseteq \partial \ol{B}_R$.  From a Cauchy name of $R > 0$, it is possible to compute an enumeration of $\{i\ :\ I_i \subseteq \R\setminus\ol{B}_R\}$.  

Since the boundary of $C_{k_0}$ consists of finitely many rational numbers, from a Cauchy name of a positive real $R$, it is possible to compute an enumeration of 
$\{i\ :\ I_i \subseteq \R\setminus C_R\}$.  Hence, $R \mapsto \mu(C_R)$ is upper semi-computable.

So, given an open rational interval $I \subseteq (0, \infty)$, it is possible to compute $R \in I$ so that $\mu(C_R) = 0$.  
Thus, we can compute $R_0 >0$ so that $\mu(\ol{B}_{R_0}) < r$ and $\mu(C_{R_0}) = 0$.  Hence, $\mu(\partial B_{R_0})) = 0$.  We then compute $N_0 \in \N$ so that $2^{-N_0} < r - \mu(\ol{B}_{R_0})$.  

$\ol{B}_{R_0}$ is a finite union of pairwise disjoint closed intervals and singletons.  The endpoints of these intervals are computable as are these singletons.  Furthermore, their $\mu$-measure is $0$.  Thus, $\ol{B}_{R_0}$ is $\mu$-almost decidable.  
So, we can now compute $n_0 \in \N$ so that $|\mu_n(\ol{B}_{R_0}) - \mu(\ol{B}_{R_0})| < 2^{-N_0}$ when $n \geq n_0$.  Set $g(r) = n_0$.  Thus, when $n \geq g(r)$, $\mu_n(\ol{B}_{R_0}) < r$ and so $\mu_n(C) \leq \mu_n(C_{k_0}) \leq \mu_n(\ol{B}_{R_0}) < r$.  Thus, $g$ is a witness that $\limsup_n \mu_n(C)$ is not larger than $\mu(C)$.

We note that an index of $\ol{B}_{R_0}$ as a $\mu$-almost decidable set can be computed from an index of $C$.  Thus, the construction of $g$ from an index of $C$ is uniform.
\end{proof}

\section{Conclusion}

We introduced two effective notions of weak convergence of measures in $\M(\R)$.  
In the first, moduli of convergence are produced for computable functions in $C_b(\R)$.   In the second, 
moduli of convergence are produced for all functions in $C_b(\R)$ via c.o. names.  While the second appears more powerful, Theorem \ref{thm:ewc.equiv} demonstrates that the two are in fact equivalent.
By means of this equivalence, we proved an effective version of the Portmanteau Theorem. 
Altogether, Theorems \ref{thm:ewc.equiv} and \ref{thm:ept} provide a broad characterization of effective weak convergence in $\M(\R)$. 

We could have also started our development by defining effective weak convergence as effective convergence in the Prokhorov metric which is known to yield a computable metric space on $\mathcal{P}(\R)$ \cite{HR09b}.  
However, such a definition would not provide much utility for the development of an effective theory of weak convergence unless first proven equivalent to the conditions we have set forth.  This equivalence will be investigated in a forthcoming paper.  

\bibliographystyle{amsplain}
\bibliography{ourbib}

\end{document}